\documentclass[sn-aps]{sn-jnl}

\usepackage{amsmath,amssymb,amsfonts}%
\usepackage{amsthm}%
\usepackage{mathrsfs}%
\usepackage[title]{appendix}%
\usepackage{xcolor}%
\usepackage{textcomp}%
\usepackage{manyfoot}%
\usepackage{booktabs}%
\usepackage{algorithm}%
\usepackage{algorithmicx}%
\usepackage{algpseudocode}%
\usepackage{listings}%

\newtheorem{thm}{Theorem}[section]

\newtheorem{ex}{Example}[section]%
\newtheorem{rem}{Remark}[section]%

\newtheorem{defn}{Definition}[section]%
\newtheorem{lem}{Lemma}[section]%
\newtheorem{cor}{Corollary}[section]

\begin{document}

\title[Article Title]{A $W$-weighted generalization of $\{1,2,3,1^{k}\}$-inverse for rectangular matrices}


\author*[1]{\fnm{Geeta} \sur{Chowdhry}}\email{geetac.207ma003@nitk.edu.in}
\author[2]{\fnm{Falguni} \sur{Roy}}\email{royfalguni@nitk.edu.in}


\affil*[1,2]{\orgdiv{Department of Mathematical and Computational Sciences}, \orgname{National Institute of Technology Karnataka, Surathkal}, \orgaddress{\street{NH 66, Srinivasnagar 
Surathkal}, \city{Mangalore}, \postcode{575025}, \state{Karnataka}, \country{India}}}



\abstract{This paper presents a novel extension of the $\{1,2,3,1^{k}\}$-inverse concept to complex rectangular matrices, denoted as a $W$-weighted $\{1,2,3,1^{k}\}$-inverse (or $\{1',2',3',{1^{k}}'\}$-inverse), where the weight $W \in \mathbb{C}^{n \times m}$. The study begins by introducing a weighted $\{1,2,3\}$-inverse (or $\{1',2',3'\}$-inverse) along with its representations and characterizations. The paper establishes criteria for the existence of $\{1',2',3'\}$-inverses and extends the criteria to $\{1'\}$-inverses.
It is further demonstrated that $A\in \mathbb{C}^{m \times n}$ admits a $\{1',2',3',{1^{k}}'\}$-inverse if and only if $r(WAW)=r(A)$, where $r(\cdot)$ is the rank of a matrix. The work additionally establishes various representations for the set $A\{ 1',2',3',{1^{k}}'\}$, including canonical representations derived through singular value and core-nilpotent decompositions. This, in turn, yields distinctive canonical representations for the set $A\{ 1,2,3,{1^{k}}\}$. $\{ 1',2',3',{1^{k}}'\}$-inverse is shown to be unique if and only if it has index $0$ or $1$, reducing it to the weighted core inverse. Moreover, the paper investigates properties and characterizations of $\{1',2',3',{1^{k}}'\}$-inverses, which then results in new insights into the characterizations of the set $A\{ 1,2,3,{1^{k}}\}$.}




\maketitle

\section{Introduction and preliminaries}\label{sec1}
Some essential generalized inverses in matrix theory are the Moore-Penrose inverse \cite{penrose1955generalized}, the Drazin inverse \cite{drazin1958pseudo}, the group inverse \cite{drazin1958pseudo}, the weighted Moore-Penrose inverse \cite{wang2018generalized}, the Bott-Duffin inverse \cite{bott1953algebra}, and the generalized  Bott-Duffin inverse \cite{yonglin1990generalized}. First, we recall the definitions of some of these well-known generalized inverses. Prior to that, some notations which will be used throughout the paper are the following. The set of all complex matrices of order $m \times n$ is denoted by $\mathbb{C}^{m \times n}$. The set of all $m \times n$ complex matrices of rank $r$ is denoted by $\mathbb{C}_r ^{m \times n}$, and $I_n$ represents the identity matrix of order $n$. For $A \in \mathbb{C}^{m \times n}$, the symbols $A^{-1}$, $A^*$, $r(A)$, $\mathcal{R}(A)$, $\mathcal{N}(A)$, will denote the inverse, conjugate transpose, rank,  range space, null space of $A$, respectively. For $A \in \mathbb{C}^{m \times n} $ following are the four Penrose equations \cite{penrose1955generalized}
\begin{equation*}
    (1) \ AXA=A,\  (2) \ XAX=X, \  (3) \ (AX)^*=AX, \  (4) \ (XA)^*=XA. 
\end{equation*}
If $X$ is a solution to the $(i), \ (j), \ldots$ equations, then it is referred to as an $\{ i,j,...\} $-inverse of $A$, and the set of all $\{i,j,...\}$-inverses of $A$ is denoted as $A\{i,j,...\}$. Moreover, the set of all $\{i,j,...\}$-inverses of $A$ having rank $s$ is denoted as $A\{i,j,...\}_s$. An element of the set $A\{i,j,...\}$ is often denoted as $A^{(i,j,...)}$. In particular, the well known Moore-Penrose inverse of $A \in \mathbb{C}^{m \times n} $ is the $\{1,2,3,4\}$-inverse of $A$, which is unique and it is denoted as $A^{\dagger}$ \cite{penrose1955generalized}. The generalized inverses satisfying equation $(2)$ of the Penrose equations are called outer inverses. However, if $X$ satisfies equation $(1)$ of the Penrose equations, then $X$ is known as an inner inverse of $A$. For any square matrix $A \in \mathbb{C}^{n \times n} $, the smallest non-negative integer $k$ such that $r(A^k)=r(A^{k+1})$ is known as the index of $A$, and it is represented by the symbol ind$(A)$. The Drazin inverse $A^D$ of $A \in \mathbb{C}^{n \times n}$ is the unique solution of the following equations \cite{drazin1958pseudo}
\begin{equation*}
     (1^k) \ XA^{k+1}=A^{k},\  (2) \ XAX=X, \  (5) \ AX=XA,
\end{equation*}
where $k \geq \text{ind}(A)$. When $k=1$, the Drazin inverse is known as the Group inverse and is denoted by $A^{\#}$.
Drazin inverse is defined only for complex square matrices. To extend it to complex rectangular matrices, Cline and Greville in \cite{cline1980drazin} introduced the $W$-weighted Drazin inverse.
For $A \in \mathbb{C}^{m \times n}$, $W \in \mathbb{C}^{n \times m}$ the unique solution $X \in \mathbb{C}^{m \times n}$ of the following equations is the $W$-weighted Drazin inverse(WDI) of $A$ if \cite{cline1980drazin}
\begin{equation*}
    ({1^{k}}') \ XW(AW)^{k+1}=(AW)^{k},\  (2') \ XWAWX=X, \  (5') \ AWX=XWA,
\end{equation*}
where $k \geq \text{max}\{\text{ind}(AW), \  \text{ind}(WA)\}=\text{ind}(A,W)$. The WDI will be denoted by $A^{D,W}$. Particularly, for $k=1$, WDI reduces to the weighted group inverse, $A^{\#,W}$. For $m=n$ and $W=I$, the WDI reduces to the Drazin inverse of $A$. Following are the two representations of the WDI \cite{cline1980drazin}
\begin{equation}
    A^{D,W}=A[(WA)^D]^2=[(AW)^D]^2A.
\end{equation}
A generalization of the Drazin inverse for bounded linear operators between Banach spaces is given in \cite{rakovcevic2002weighted}. When the Banach spaces are finite-dimensional, this gives another version of the Drazin inverse for complex rectangular matrices, which is of our interest. Let $A,B \in \mathbb{C}^{m \times n}$ and $W \in \mathbb{C}^{n \times m}$. The $W$-product \cite{rakovcevic2002weighted, ferreyra2018revisiting, ferreyra2019weak} of $A$ and $B$ is defined as $A \star B=AWB$, and the $W$-product of $A$ with itself $l$ times is denoted by $A^{\star l}$. Moreover if $\|A\|_{W}=\|A\|\|W\|$, where $\|\cdot\|$ represents any matrix norm on $\mathbb{C}^{m \times n}$, then $(\mathbb{C}^{m \times n}, \star , \|\cdot\|_W)=\mathfrak{A}$ forms a Banach algebra with $A^{\star l}= (AW)^{l-1}A=A(WA)^{l-1}, \ l \in \mathbb{N}.$ Let $A \in \mathfrak{A}$, then the $W$-Drazin inverse \cite{rakovcevic2002weighted} of $A$ is the unique solution $X \in \mathfrak{A}$ of the following equations
\begin{equation*}
    (1^{k*}) \ X \star A^{\star (k+2)} =A^{\star (k+1)},  \ (2^*) \ X \star A \star X = X, \ (5^*) \ A \star X =X \star A,
\end{equation*}
where $k$ is some non-negative integer.

\begin{rem}
    For $A,X \in \mathbb{C}^{m \times n}$ and $\ W \in \mathbb{C}^{n \times m}$, using the definition of $W$-product, system of equations $(1^{k*},2^*,5^* )$ can be rewritten as
    \begin{equation*}
        (1^{k*}) \ X(WA)^{k+2}=(AW)^kA, \ (2^*) \ XWAWX=X, \  (5^*) \ AWX=XWA.
    \end{equation*}
    Note that the $W$-weighted Drazin inverse also satisfies the system of equations $(1^{k*},2^*,5^* )$ and  the system of equations $(1^{k*},2^*,5^* )$ has a unique solution hence the $W$-Drazin inverse is same as the $W$-weighted Drazin inverse.
\end{rem}
Furthermore, Baksalary and Trenkler developed the notion of the core inverse for square matrices in \cite{baksalary2010core}.
For $A \in \mathbb{C}^{n \times n}$ and $\text{ind}(A) \leq 1$, the core inverse of $A$ is the unique matrix $X$ such that
\begin{equation*}
    AX=AA^{\dagger}, \ \mathcal{R}(X) \subseteq \mathcal{R}(A).
\end{equation*}
Core inverse exists for all matrices of index one, and it is unique. In this paper, we mainly deal with the generalizations of the core inverse. 
The core inverse is used to solve several problems involving matrix equation solvability and matrix partial ordering \cite{baksalary2010core}. 
Various generalizations of the core inverse can be found in \cite{baksalary2014generalized, malik2014new, manjunatha2014core, ferreyra2021weak, ferreyra2018revisiting, ferreyra2019weak,wu20221, behera2020further, mosic2022extensions}. 
Since the core inverse exists only for matrices of index 0 or 1, there are several possibilities for generalizations to matrices of arbitrary index \cite{baksalary2014generalized, malik2014new, manjunatha2014core, ferreyra2021weak,wu20221}.
One such generalization is the core-EP inverse, introduced in 
\cite{manjunatha2014core} for square complex matrices of any arbitrary index.
The unique matrix $X$ is the core-EP inverse of $A$ if it satisfies
\begin{equation*}
    XAX=X, \ \mathcal{R}(X)=\mathcal{R}(X^*)=\mathcal{R}(A^k),
\end{equation*}
where $k=\text{ind}(A)$ and it is denoted by $A^\text{\textcircled{$\dagger$}}$.
Ferreyra et al. \cite{ferreyra2018revisiting} extended the core-EP inverse to matrices of any order. For $A \in \mathbb{C}^{m \times n},\ W \in \mathbb{C}^{n \times m}$, the weighted core-EP inverse, $A^{ \text{\textcircled{$\dagger$},$W$}}$ is the unique solution $X$ to the following equations
\begin{equation*}
    WAWX=(WA)^k[(WA)^k]^{\dagger}, \ \mathcal{R}(X) \subseteq \mathcal{R}((AW)^k),
\end{equation*}
where $k \geq \text{ind}(A,W)$. 
An important characterization of the weighted core-EP inverse \cite{ferreyra2018revisiting} of $A \in \mathbb{C}^{m \times n}$ is as follows. $X \in \mathbb{C}^{m \times n}$ is the weighted core-EP inverse of $A$ if and only if it is $\{1^{k*}, 2'\}$-inverse and satisfies 
\begin{equation*}
    (3') \ (W A W X )^*= W A W X, \text{and}  \ \mathcal{R}(X) \subseteq \mathcal{R}((AW)^{ K} ).
\end{equation*}
When $k=1$, weighted core-EP inverse is known as the weighted core inverse $A^{ \text{\textcircled{$\#$},$W$}}$ of $A$ \cite{ferreyra2019weak}.
Recently, a generalization of the core inverse, possessing the $\{1\}$-inverse property, is introduced for square matrices of arbitrary index by C. Wu and J. Chen in \cite{wu20221}.
Let $A \in \mathbb{C}^{n \times n}$, then $\{1, 2, 3, 1^{k}\}$-inverse \cite{wu20221} of $A$ is defined as an element of the set $A\{1,2,3,1^{k} \}$, where $k$ is a non-negative integer and is denoted by $A^{\text{{\textcircled{$-$}}}}$.
Also, $X$ is a $\{1, 2, 4, {}^{k}1\}$-inverse of $A$ if it satisfies the following equations \cite{wu20221}
\begin{equation*}
     X \in A\{1,2,4\}, \  \text{and} \  ({}^{k}1) \ A^{k+1}X=A^k,
\end{equation*}
where $k$ is a non-negative integer. This is the only generalization of the core inverse of $A$ among \cite{baksalary2014generalized, malik2014new, manjunatha2014core, ferreyra2021weak}, which is also a $\{1\}$-inverse or inner inverse of $A$. C. Wu and J. Chen \cite{wu20221} characterized the set $A\{1,2,3,1^k\}$ and gave a canonical representation of the elements of the set. \\
Our work primarily focuses on the generalizations of $\{1, 2, 3, 1^{k}\}$-inverse. Since $\{1, 2, 3, 1^{k}\}$-inverse of $A$ is defined for square matrices, a generalization of the $\{1, 2, 3, 1^{k}\}$-inverse for complex rectangular matrices is introduced and studied in this paper. Recently, weighted inner inverses for complex rectangular matrices were defined in \cite{behera2022weighted}. $X \in \mathbb{C}^{m \times n}$ is weighted inner inverse of $A \in \mathbb{C}^{m \times n}$ if it satisfies the equation $(1') \ AWXWA=A$. Motivated by the definitions of weighted inner inverse \cite{behera2022weighted}, $W$-weighted Drazin inverse \cite{cline1980drazin} and weighted core-EP inverse \cite{manjunatha2014core}, 
an extension of $\{ 1,2,3,1^{k}\}$-inverse is defined for $A\in \mathbb{C}^{m \times n}$ as $W$-weighted $\{ 1,2,3,1^{k}\}$-inverse using the weight $W \in \mathbb{C}^{n \times m}$, which is denoted as $\{ 1',2',3',{1^{k}}'\}$-inverse. It is shown that any $A \in \mathbb{C}^{m \times n}$ have a $\{ 1',2',3',{1^{k}}'\}$-inverse if and only if $r(WAW)=r(A)$. The sets $A\{ 1',2',3'\}$ and $A\{ 1',2',3',{1^{k}}'\}$ are completely characterized. Canonical representations of the elements of the set $A\{ 1',2',3',{1^{k}}'\}$ are obtained using the core-nilpotent and the singular value decompositions. A generalization of the core-nilpotent decomposition can be seen in \cite{KMP2021gen}. $\{ 1',2',3',{1^{k}}'\}$-inverse of $A$ is unique if and only if it has index $0$ or $1$ and in that case it reduces to $A^{ \text{\textcircled{$\#$},$W$}}$. Some properties and characterizations of $\{ 1',2',3',{1^{k}}'\}$-inverse are also obtained.\\
This paper is organized as follows. Section \ref{sec2} contains definitions of $\{1',2',3'\}$-inverse and $\{1',2',4'\}$-inverse along with various characterizations and representations of the sets $A\{1',2',3'\}$ and $A\{1',2',4'\}$. In section \ref{sec3}, $\{1',2',3', {1^{k}}'\}$-inverse of $A$ is introduced. Existence criteria, characterizations, and representations of the set  $A\{1',2',3', {1^{k}}' \}$ are obtained using the results from Section \ref{sec2}. Conditions for uniqueness of $\{1',2',3',{1^{k}}'\}$-inverse are also obtained. Additionally, two canonical forms for elements of $A\{1',2',3',{1^{k}}'\}$ are obtained using the singular value and the core-nilpotent decompositions of matrices. Lastly, Section \ref{sec4} contains some properties and characterizations of $\{1',2',3', {1^{k}}'\}$-inverse of $A$.

\section{$\{1',2',3' \}$ and $\{1',2',4' \}$-inverses}\label{sec2}
In this section, the sets $A\{1',2',3'\}$ and $A\{1',2',4'\}$ are defined and investigated. The criteria for non-emptiness, various characterizations, and representations of these sets are also obtained.

\begin{defn} \label{def1}
Let $A \in \mathbb{C}^{m \times n}, \ W \in \mathbb{C}^{n \times m}$. Then $X \in \mathbb{C}^{m \times n}$ is defined as a $\{1',2',3'\}$-inverse of $A$ if it satisfies the following equations
\begin{equation*}
    (1') AWXWA=A, \ (2') XWAWX=X, \   (3')(WAWX)^*=WAWX.
\end{equation*}
\end{defn}
\begin{defn}
Let $A \in \mathbb{C}^{m \times n}, \ W \in \mathbb{C}^{n \times m}$. Then $X \in \mathbb{C}^{m \times n}$ is defined as a $\{1',2',4'\}$-inverse of $A$ if it satisfies the equations
\begin{equation*}
    (1') AWXWA=A, \ (2') XWAWX=X, \   (4')(XWAW)^*=XWAW.
\end{equation*}
\end{defn}
The existence criteria for $\{1',2',3'\}$-inverses of $A \in \mathbb{C}^{m \times n}$ is given by the following theorem.
\begin{thm}
 \label{1,2,3 non empty}
Let $A \in \mathbb{C}^{m \times n}, \ W \in \mathbb{C}^{n \times m}$, then $A\{1',2',3'\} \neq \emptyset$ if and only if $r(WAW)=r(A)$.
\end{thm}
\begin{proof}
Firstly, assume $A\{1',2',3'\} \neq \emptyset$. Then for $X \in A\{1',2',3'\} $, $r(X) \leq r(WAW) \leq r(A)\leq r(X)$. Conversely, if $r(A) = r(WAW)$ then $\mathcal{R}(A^*)=\mathcal{R}((WA)^*)$. So, there exists $U \in \mathbb{C}^{n \times m} $ such that $A^*=(WA)^*U$, which gives $A=U^*WA$. Also, $r(A) = r(WAW)$ gives $\mathcal{R}(AW)=\mathcal{R}(A)$, which implies $A=AWV$ for some $ V \in \mathbb{C}^{m \times n}$. Since $(WAW)\{1,2,3\} \neq \emptyset$,  there exists $X \in (WAW)\{1,2,3\}$ satisfying $XWAWX=X$ and $(WAWX)^*=WAWX$. Now
\begin{equation*}
    \begin{split}
    AWXWA =U^*WAWXWAWV= U^*WAWV=AWV=A.
    \end{split}
\end{equation*}
Hence, $X \in A\{1',2',3'\}$.
\end{proof}

\begin{cor} \label{A(1,2,3)=waw(1,2,3)}
    Let $A \in \mathbb{C}^{m \times n}, \ W \in \mathbb{C}^{n \times m}$ such that $r(WAW)=r(A)$. Then $X \in A\{1',2',3'\}$ if and only if $X \in (WAW)\{1,2,3\}$. 
\end{cor}
\begin{proof}
    The proof is direct from the proof of Theorem $\ref{1,2,3 non empty}$.
\end{proof}
Considering $\{1',2',4'\}$-inverses instead of $\{1',2',3'\}$-inverses in the proof of the Theorem $\ref{1,2,3 non empty}$, the existence criteria for $\{1',2',4'\}$-inverses for $A\in \mathbb{C}^{m \times n}$ is given by the following corollary.
\begin{cor}
    Let $A \in \mathbb{C}^{m \times n}, \ W \in \mathbb{C}^{n \times m}$, then $A\{1',2',4'\} \neq \emptyset$ if and only if $r(WAW)=r(A)$.
\end{cor}

Following is an example to verify the equivalence of $\{1',2',3'\}$-inverse and $(WAW)\{1,2,3\}$-inverse. Furthermore, in the case $\{1',2',3'\}$-inverse exists, it is shown that it may not be unique.
\begin{ex} 
Take $W=\begin{bmatrix}
    1 & 0\\
    0 & 0
\end{bmatrix}$, $A=\begin{bmatrix}
    0 &0 \\
    1 & 1
\end{bmatrix}$, $B=\begin{bmatrix}
    1 & 1\\
    0 & 0
\end{bmatrix}$. Then $WAW=0$ (matrix with all entries $0$), hence $r(WAW)\neq r(A)$. 
It is well-known that $$ X=((WAW)^*WAW)^{(1)}(WAW)^*$$ is $\{1,2,3\}$-inverse of $WAW$. 
Here $ X=((WAW)^*WAW)^{(1)}(WAW)^*=0$. Clearly $X=0$ doesn't satisfy equation $(1{'})$, hence $X \notin A\{1',2',3'\}$. Whereas $WBW=\begin{bmatrix}
    1 & 0\\
    0 & 0
\end{bmatrix}$, hence $r(WBW)=r(B)=1$. Then, $\displaystyle$ $Y=((WBW)^*WBW)^{(1)}(WBW)^*=\begin{bmatrix}
    1 & 0\\
    a & 0
\end{bmatrix},$
$\forall a \in \mathbb{C}$. Therefore, $Y\in B\{1',2',3'\}$ is not unique.
\end{ex}

The next result gives a representation and characterization of the set of weighted inner inverses of a complex rectangular matrix. To accomplish this, the following theorem is required.
\begin{thm} \cite{ben2003generalized} \label{AXB=D}
    Let $A \in \mathbb{C}^{m \times n}, \ B \in \mathbb{C}^{p \times q}, \ D \in \mathbb{C}^{m \times q}$. Then the matrix equation $AXB=D$ is consistent if and only if, for some $A^{(1)}, B^{(1)}$, $AA^{(1)}DB^{(1)}B=D$ holds, in which case the general solution is
    \begin{equation*}
        X=A^{(1)}DB^{(1)}+Y-A^{(1)}AYBB^{(1)},
    \end{equation*}
    where $Y \in \mathbb{C}^{n \times p}$ is arbitrary.  
\end{thm} 
\begin{thm} \label{thm2.3}
    Let $A \in \mathbb{C}^{m \times n}, \ W \in \mathbb{C}^{n \times m}$. Then the set $A\{1'\}$ is nonempty if and only if $r(AW)=r(WA)=r(A)$. If $A\{1'\} \neq \emptyset $, then
    \begin{equation*}
        A\{1'\}=\{(AW)^{(1)}A(WA)^{(1)}+Y-(AW)^{(1)}AWYWA(WA)^{(1)} \mid Y\in \mathbb{C}^{m \times n} \}.
    \end{equation*}
\end{thm}
\begin{proof}
    The proof follows by substituting $AW$ for $A$, $WA$ for $B$ and $A$ for $D$ in Theorem $\ref{AXB=D}$.
\end{proof}
\begin{rem} \label{rem2.1}
    Note that since $r(WAW)=r(A)$ is equivalent to $r(AW)=r(WA)=r(A)$ by Frobenius inequality \cite{horn2012matrix}, the set $A\{1'\}$ is nonempty if and only if the set $A\{1',2',3'\}$ is nonempty.
\end{rem}

The following result provides an improved version of Corollary 2.1 in \cite{behera2022weighted}.
\begin{thm}
    Let $A \in \mathbb{C}^{m \times n}$, $B \in \mathbb{C}^{m \times s}$, $c \in \mathbb{C}^{n \times k}$ and $W \in \mathbb{C}^{n \times m}$. Then $X$ a weighted inner inverse of $A$ satisfying $R(XWAW)=R(B)$ and $R(I_n-WAWX)=R(C)$ exists if and only if $Y$ an inner inverse of $WAW$ satisfying $R(XWAW)=R(B)$ and $R(I_n-WAWX)=R(C)$ exists and $r(WAW)=r(A)$. In this case, $X=Y$.
\end{thm}
\begin{proof}
    The proof follows from Theorem \ref{thm2.3} and Remark \ref{rem2.1}.
\end{proof}

\noindent The next theorem gives two full-rank representations of the sets $A \{ 1', 2', 3' \} $ and $A \{ 1', 2', 4' \} $. This can be done by using the following lemma, which gives the full-rank representations of $\{2,3\}, \ \{2,4\}, \ \{1,2,3\}$ and $\{1,2,4\}$-inverses.
\begin{lem} \cite{stanimirovic2011full} \label{222}
    Let $A \in \mathbb{C}^{m \times n}_r$ and $s$ be an integer such that $0 < s \leq r$. Suppose that there exist two arbitrary matrices $V$ and $U$ such that $r(VA)$=$r(V)$ and $r(AU)$=$r(U)$. Then,
\begin{enumerate}
    \item[(i)] $A\{2,3\}_s=\{U((AU)^{*}AU)^{-1}(AU)^{*} : U \in \mathbb{C}^{n \times s}_s \}$,
    \item[(ii)] $A\{2,4\}_s=\{(VA)^{*}(VA(VA)^{*})^{-1}V : V \in \mathbb{C}^{s \times m}_s\}$,
    \item[(iii)] $A\{1,2,3\}=\{U((AU)^{*}AU)^{-1}(AU)^{*} : U \in \mathbb{C}^{n \times r}_r \}=A\{2,3\}_r$,
    \item[(iv)] $A\{1,2,4\}=\{(VA)^{*}(VA(VA)^{*})^{-1}V : V \in \mathbb{C}^{r \times m}_r\}=A\{2,4\}_r$.
\end{enumerate}
\end{lem}

\begin{thm}

Let $A \in \mathbb{C}^{m \times n}, \ W \in \mathbb{C}^{n \times m}$ and $r=r(A)$, then
\begin{equation*} 
\begin{split}
    A\{1',2',3'\} =&\{Y((WAWY)^*WAWY)^{-1}(WAWY)^*
    \mid Y \in \mathbb{C}^{m \times r}_r, \\& \  r(WAWY)=r(Y)=r(A) \} \\
    =& \{Y(WAWY)^{\dagger} \mid Y \in \mathbb{C}^{m \times r}_r, \  r(WAWY)=r(Y)=r(A) \}
    \end{split}
\end{equation*}
and
\begin{equation*}
\begin{split}
    A\{1',2',4'\} =&\{(YWAW)^*(YWAW(YWAW)^*)^{-1}Y \mid Y \in \mathbb{C}^{r \times m}_r, \\& \  r(YWAW)=r(Y)=r(A) \} \\
    =& \{(YWAW)^{\dagger}Y \mid Y \in \mathbb{C}^{r \times m}_r, \  r(YWAW)=r(Y)=r(A) \}.
    \end{split}
\end{equation*}
\end{thm}
\begin{proof}
Using Lemma $\ref{222}$, we have \begin{equation} \label{2,3}
    \begin{split}
        &(WAW)\{2,3\}_s = \\
        &\{Y((WAWY)^*WAWY)^{-1}(WAWY)^* \mid   Y \in \mathbb{C}^{m \times r}_s, \  r(WAWY)=r(Y) \},
    \end{split}
\end{equation}
    for $s\leq r(WAW)$. If $s=r(WAW)$ then $(WAW)\{2,3\}_s=(WAW)\{1,2,3\}$. If $X \in (WAW)\{1,2,3\} $, then by equation $(\ref{2,3})$, $r(X)= r(Y)= r(WAWY) \leq 
 r(WAW) \leq r(A)$. Now if $r(A)=r(X)$, then $s=r(WAW) = r(A)=r$. Hence by Corollary \ref{A(1,2,3)=waw(1,2,3)}, $A\{1', 2',3'\} \neq \emptyset$ and $X \in A\{1', 2',3'\}$. Now if $X' \in A\{1', 2',3'\} \neq \emptyset$, then from Definition \ref{def1},  $r(X')=r(A)$ and by Corollary \ref{A(1,2,3)=waw(1,2,3)}, $X' \in (WAW)\{1,2,3\}$. The proof follows similarly for the set $A\{1',2',4'\}$.
\end{proof}
The following lemma provides a representation of $\{1,2,3\}$ and $\{1,2,4\}$-inverses. A similar representation is obtained for $\{1',2',3'\}$ and $\{1',2',4'\}$-inverses in the next theorem.
\begin{lem} \cite{ben2003generalized} Let $A \in \mathbb{C}^{m \times n}$, then
\begin{equation*}
\begin{split}
    A\{1,2,3\}=\{A^{\dagger}+ (I_m-A^{\dagger}A)UA^{\dagger} \mid U \in \mathbb{C}^{m \times m} \},
 \\
  A\{1,2,4\}=\{A^{\dagger}+ A^{\dagger}V(I_n-AA^{\dagger}) \mid V \in \mathbb{C}^{m \times m} \}.
\end{split}
\end{equation*}
\end{lem}
\begin{thm}

 \label{lemma 7}
    Let $A \in \mathbb{C}^{m \times n}, \ W \in \mathbb{C}^{n \times m}$, then
\begin{equation*}
\begin{split}
    A\{1',2',3'\}=\{(WAW)^{\dagger}+ (I_m-(WAW)^{\dagger}WAW)U(WAW)^{\dagger} \\ \mid U \in \mathbb{C}^{m \times m}, \  r(WAW)=r(A) \}
    \end{split}
\end{equation*}
and
\begin{equation} \label{4''}
\begin{split}
    A\{1',2',4'\}=\{(WAW)^{\dagger}+ (WAW)^{\dagger}V(I_n-WAW(WAW)^{\dagger}) \\ \mid V \in \mathbb{C}^{m \times m}, \  r(WAW)=r(A) \}.
    \end{split}
\end{equation}
\end{thm}
\begin{proof}
Let $X \in A\{1',2',3'\}$ be an arbitrary element of the set $A\{1',2',3'\}$. Substitute $XWAW$ for $U$ in $(WAW)^{\dagger}+ (I_m-(WAW)^{\dagger}WAW)U(WAW)^{\dagger}$, we get $(WAW)^{\dagger}+ (I_m-(WAW)^{\dagger}WAW)XWAW(WAW)^{\dagger}=X$. The converse is trivial. Likewise equation $(\ref{4''})$ follows by taking $V=WAWX$.
\end{proof}

It is well known that $X \in \mathbb{C}^{m \times n}$ is a $\{1,2,3\}$-inverse(or $\{1,2,4\}$-inverse) of $A \in \mathbb{C}^{m \times n}$ if and only if it is an outer inverse such that $\mathcal{N}(X)=\mathcal{N}(A^*)$(or $\mathcal{R}(X)=\mathcal{R}(A^*)$).
The following two theorems characterizes the sets $A\{1',2',3'\}$ and $A\{1',2',4'\}$. To prove these theorems, we present the following lemma.
\begin{lem}  \label{any 2 give 3} \cite{ben2003generalized}
   A third assertion is implied by any two of the following three:
   \begin{itemize}
       \item[(i)] $X \in A\{1\}$.
       \item[(ii)] $X \in A\{2\}$.
       \item[(iii)] $r(X)=r(A)$.
   \end{itemize}
\end{lem}
\begin{thm}
    
    \label{a2ts}
Let $A \in \mathbb{C}^{m \times n}, \ W \in \mathbb{C}^{n \times m}$. Then the following statements are equivalent.
\begin{itemize} 
    \item[(i)] $X \in A\{1',2',3'\}$.
    \item[(ii)]  $r(WAW)=r(A)$, $XWAWX=X$,  and $\mathcal{N}(X)=\mathcal{N}((WA)^*)$.
    \item[(iii)] $r(WAW)=r(A)$, $AWXWX=A$, and $\mathcal{N}(X)=\mathcal{N}((WA)^*)$.
\end{itemize}

\end{thm} 
\begin{proof}
    $(i) \iff (ii)$. First, $X \in A\{1',2',3'\}$ implies $X \in (WAW)\{1,2,3\}$. Hence, $XWAWX=X$ and $\mathcal{N}(X)=\mathcal{N}((WAW)^*)$. Also $X \in A\{1',2',3'\}$ implies $r(WAW)=r(A)$ whence $\mathcal{N}((WAW)^*)=\mathcal{N}((WA)^*)$. Converse is trivial.\\
    $(ii) \iff (iii)$. It is direct by Lemma $\ref{any 2 give 3}$.
\end{proof}
\begin{thm}

Let $A \in \mathbb{C}^{m \times n}, \ W \in \mathbb{C}^{n \times m}$. Then the following statements are equivalent.
\begin{itemize} 
    \item[(i)] $X \in A\{1',2',4'\}$.
    \item[(ii)]  $r(WAW)=r(A)$, $XWAWX=X$,  and $\mathcal{R}(X)=\mathcal{R}((AW)^*)$.
    \item[(iii)] $r(WAW)=r(A)$, $AWXWX=A$, and $\mathcal{R}(X)=\mathcal{R}((AW)^*)$.
\end{itemize}

\end{thm} 
\begin{proof}
    $(i) \iff (ii)$. Clearly, $X \in A\{1',2',4'\}$ implies $X \in (WAW)\{1,2,4\}$. Hence, $XWAWX=X$ and $\mathcal{R}(X)=\mathcal{R}((WAW)^*)$. Also $X \in A\{1',2',4'\}$ implies $r(WAW)=r(A)$ which further gives $\mathcal{R}((WAW)^*)=\mathcal{R}((AW)^*)$. Converse is trivial.\\
    $(ii) \iff (iii)$. It is direct by Lemma $\ref{any 2 give 3}$.
\end{proof}

\section{$\{1',2',3', {1^{k}}' \}$-inverses}\label{sec3}
In this section, we begin by introducing $\{1',2',3', {1^{k}}' \}$-inverses, which are defined as $W$-weighted $\{1,2,3,1^{k}\}$-inverse of $A \in \mathbb{C}^{m \times n}$. The existence criteria, characterizations, and representations of the set  $A\{1',2',3', {1^{k}}' \}$ are obtained using the results of $ A \{1', 2', $ $3' \}$ from Section \ref{sec2}. Conditions for uniqueness of the $\{1',2',3',{1^{k}}'\}$-inverse are also obtained. Canonical representations of the elements of the set $A\{ 1',2',3',$ ${1^{k}}'\}$ are obtained using the singular value decomposition and the core-nilpotent decomposition of matrices.
\begin{defn} \label{def}
Let $A \in \mathbb{C}^{m \times n}, \ W \in \mathbb{C}^{n \times m}$. Then $X \in \mathbb{C}^{m \times n}$ is called a $W$-weighted $\{1,2,3,1^k\}$-inverse if it satisfies the following equations
\begin{equation*}
\begin{split}
    &(1') AWXWA=A, \ (2') XWAWX=X, \\   &(3')(WAWX)^*=WAWX, \    ({1^{k}}') (XW)(AW)^{k+1}=(AW)^k,
    \end{split}
\end{equation*}
where $k \geq \text{ind}(A,W)$.
\end{defn}
Dually, $X \in \mathbb{C}^{m \times n}$ is a $W$-weighted $\{1,2,4,{}^{k}1\}$-inverse if it satisfies the following equations
\begin{equation*}
\begin{split}
    (1')& AWXWA=A, \ (2') XWAWX=X, \  \\ (4')& (XWAW)^*=XWAW, \ ({{}^{k}1}')
    (AW)^{k+2}X=(AW)^kA,
    \end{split}
\end{equation*}
where $k \geq \text{ind}(A,W)$. Note that for the dual case, the equation $({{}^{k}1}')$ is motivated by the definition of $W$-Drazin inverse \cite{rakovcevic2002weighted}.
A $W$-Weighted $\{1,2,3,1^k\}$-inverse of $A$ will be denoted by $A^{\text{\textcircled{$-$}},W}$ and a $W$-Weighted $\{1,2,4,{}^{k}1\}$-inverse of $A$ will be denoted by $A_{\text{\textcircled{$-$}},W}$.

The existence and a representation of $\{1',2',3',{1^{k}}'\}$-inverses of $A$ are given by the following result.
\begin{thm} \label{exist amd rep}
Let $A \in \mathbb{C}^{m \times n}, \ W \in \mathbb{C}^{n \times m}$ such that $k = \text{ind}(A, W)$ and $r(WAW)=r(A)$. Then
\begin{equation} \label{1,2,3,1m}
\begin{split}
     A\{1',2',3',{1^{k}}'\} = &\{X+(I_m-XWAW)A^{D,W}WAWX \mid X \in A\{1',2',3'\}\} \\
    =& \{X+(I_m-XWAW)YWAWX \mid X \in A\{1',2',3'\}, \\& Y \in A\{{1^{k}}'\}\}
    \end{split}
\end{equation}  
and
\begin{equation} \label{4'}
\begin{split}
    A\{1',2',4',{{}^{k}1}'\} =& \{X+XWAWA^{D,W}(I_n-WAWX) \mid X \in A\{1',2',4'\}\} \\
    =& \{X+XWAWY(I_n-WAWX) \mid X \in A\{1',2',4'\}, \\& Y \in A\{{{}^{k}1}'\}\}.
    \end{split}
\end{equation}
Therefore, $A\{1',2',3',{1^{k}}'\}\neq \emptyset$ (or $A\{1',2',4',{{}^{k}1}'\}\neq \emptyset$) if and only if $r(WAW)=r(A)$.
\end{thm}
\begin{proof}
    To get equation $(\ref{1,2,3,1m})$, the proof will be given in the following three steps:
    \begin{itemize}
        \item[{(i)}] $A\{1',2',3',{1^{k}}'\} \subseteq \{X+(I_m-XWAW)A^{D,W}WAWX \mid X \in A\{1',2',3'\}\}$.
        \item[{(ii)}] $\{X+(I_m-XWAW)A^{D,W}WAWX \mid  X \in A\{1',2',3'\}\} \} \subseteq \{X+(I_m-XWAW)YWAWX \mid X \in A\{1',2',3'\}, Y \in A\{{1^{k}}'\}\} $.
        \item[{(iii)}] $\{X+(I_m-XWAW)YWAWX \mid X \in A\{1',2',3'\}, Y \in A\{{1^{k}}'\}\} \subseteq A\{1',2',3',{1^{k}}'\}$.
    \end{itemize}
    Proofs of the steps.
    \begin{itemize}
        \item[(i)] Suppose that $Z \in A\{1',2',3',{1^{k}}'\}$. Then,
    \begin{equation*}
    \begin{split}
    (I_m-ZWAW)A^{D,W}WAWZ &= (I_m-ZWAW)(A^{D,W}W)^k(AW)^kZ \\
    &= ((AW)^k-ZWAW(AW)^k)(A^{D,W}W)^kZ \\
    &= 0.
    \end{split}
    \end{equation*}
    Hence, $Z=Z+(I_m-ZWAW)A^{D,W}WAWZ$, implying
    $Z \in \{X+(I_m-XWAW)A^{D,W}WAWX \mid X \in A\{1',2',3'\}\}$.
    \item[(ii)] The proof implies from the fact $A^{D,W} \in A\{{1^{k}}'\}$.
    \item[(iii)] Let $X \in A\{1',2',3'\}, Y \in A\{{1^{k}}'\}$ and $Z=X+(I_m-XWAW)YWAWX $. So, $AWZ=AWX$ implies $Z \in A\{1',2',3'\} $. Now,
    \begin{equation*}
    \begin{split}
         (ZW)(AW)^{k+1} =& (X+(I_m-XWAW)YWAWX)W(AW)^{k+1} \\
         =& XW(AW)^{k+1}+YWAWXW(AW)^{k+1} \\&-XWAWYWAWXW(AW)^{k+1} \\
         =& YWAWXW(AW)^{k+1}\\
         =& YWAW(AW)^{k}\\
         =& (AW)^k.
    \end{split}
    \end{equation*}
     Hence, $Z \in A\{1',2',3',{{1^{k}}'} \} $. 
    \end{itemize}
    Likewise equation $(\ref{4'})$ can be proved.
\end{proof}

\begin{rem} \label{remark 13}
$X=(WAW)^{\dagger}$ in Theorem $\ref{exist amd rep}$, gives a particular representation of $A^{\text{\textcircled{$-$}},W}$,  i.e., $$ A^{\text{\textcircled{$-$}},W} = (WAW)^{\dagger} +  (I_m-(WAW)^{\dagger}WAW) A^{D,W}WAW(WAW)^{\dagger}.$$
\end{rem}
The following corollary gives the sufficient condition for the set $A\{1',2',$ $3',{1^{k}}'\}$ to be a singleton set and hence the uniqueness of $\{1',2',3',{1^{k}}'\}$-inverse of $A$.
\begin{cor} \label{1,2,3,11}
Let $A \in \mathbb{C}^{m \times n}, \ W \in \mathbb{C}^{n \times m}$ such that $ k=\text{ind}(A,W) \leq 1$ and $r(WAW)=r(A)$. Then $A^{\text{\textcircled{$\#$}},W}$(weighted core inverse) is the unique $\{1',2',3',{1^{1}}'\}$-inverse of $A$.
\end{cor}
\begin{proof}
    Let $X \in A\{1',2',3',{1^{1}}'\}$, thus there exists $Z \in A\{1',2',3'\}$ by Theorem \ref{exist amd rep}, such that 
    \begin{equation}\label{eq 18}
        \begin{split}
            X &= Z+(I_m-ZWAW)A^{D,W}WAWZ\\
            &= Z+A^{D,W}WAWZ-ZWAWA^{D,W}WAWZ\\
            &=Z+(A(WA)^{\#}(WA)^{\#})WAWZ-ZWAW(A(WA)^{\#}(WA)^{\#})WAWZ\\
            &=Z+A(WA)^{\#}WZ-ZWAWZ\\
            &= A(WA)^{\#}WZ.
        \end{split}
    \end{equation}
    Since for $k=1$, $A^{\text{\textcircled{$\#$}},W}=A^{\text{\textcircled{$\dagger$}}, W}$ \cite{ferreyra2019weak}, we will show that $X=A^{\text{\textcircled{$\dagger$}}, W}$. It is clear that $(WAWX)^2=WAWX$ and $X \in A\{3'\}$ imply that $WAWX$ is an orthogonal projection. Clearly $\mathcal{R}(WAWX) \subseteq \mathcal{R}(WA)$. Now since $X \in A\{1',2'\}$, it follows,
    \begin{equation*}
    \begin{split}
        r(A)&= r(AWXWA)\leq r(X)=r(XWAWX) \\& \leq r(WAWX) \leq r(WA) \leq r(A).
        \end{split}
    \end{equation*}

Thus, $r(WAWX)=r(WA)$ and $\mathcal{R}(WAWX)=\mathcal{R}(WA)$. Moreover, since $$X = A(WA)^{\#}W  Z = A(WA)^{\#}WA(WA)^{\#} WZ = AWA(WA)^{\#} (WA)^{\#} WZ,$$ it follows
    \begin{equation*}
        \mathcal{R}(X)=\mathcal{R}(AWA(WA)^{\#}(WA)^{\#}WZ) \subseteq \mathcal{R}(AW).
    \end{equation*}
    Hence $X=A^{\text{\textcircled{$\dagger$}}, W}$, which is unique.
\end{proof}
Note that equation (\ref{eq 18}) of Corollary $\ref{1,2,3,11}$ also gives a new representation of the weighted core inverse of $A$, provided $r(WAW)=r(A)$.

\begin{cor}
Let $A \in \mathbb{C}^{m \times n}, \ W \in \mathbb{C}^{n \times m}$ such that $r(WAW)=r(A)$. Then $X \in \mathbb{C}^{m \times n}$ is a weighted core inverse of $A$ if and only if 
\begin{equation*}
\begin{split}
    &AWXWA=A, \ XWAWX=X, \\& (WAWX)^*=WAWX, \ (XW)(AW)^2=AW.
    \end{split}
\end{equation*}
\end{cor}
\begin{proof}
    The proof is direct by Corollary $\ref{1,2,3,11}$ and Definition $\ref{def}$.
\end{proof}
To get the necessary condition for the uniqueness of $\{1',2',3',{1^{k}}'\} $-inverse of $A$, the following representation of the set is useful.
\begin{thm} \label{full set}
 Let $A \in \mathbb{C}^{m \times n}, \ W \in \mathbb{C}^{n \times m}$ such that $k= \text{ind}(A,W)$. Then
\begin{equation*} 
\begin{split}
    A&\{1',2',3',{1^{k}}'\}=\{A^{\text{\textcircled{$-$}}, W}+(I_m-(WAW)^{\dagger}WAW)U((WAW)^{\dagger} \\&-(WAW)^{\dagger}WAWA^{D,W}WAW(WAW)^{\dagger}) 
    \mid U \in \mathbb{C}^{m \times m}, r(WAW)=r(A) \}
    \end{split}
\end{equation*}
and
\begin{equation*}
\begin{split}
    A\{1',2'&,4',{{}^{k}1}'\}= \{A_{\text{\textcircled{$-$}},w}+((WAW)^{\dagger}-(WAW)^{\dagger}WAWA^{D,W}WAW \\ & (WAW)^{\dagger})  V(I_m -WAW(WAW)^{\dagger}) 
    \mid V \in \mathbb{C}^{m \times m}, r(WAW)=r(A) \}. 
    \end{split}
\end{equation*}

\end{thm}
\begin{proof} $\displaystyle$  By Theorem \ref{lemma 7}, any element of the set $A\{1',2',3'\}$, say $Y$, can be written in the form $Y=(WAW)^{\dagger}+ (I_m-(WAW)^{\dagger} WAW) U (WAW)^{\dagger}$ for some $U \in \mathbb{C}^{m\times m}$. Substituting $Y$ for $X$ into $X + (I_m-XWAW )A^{D,W}WAW X $, which is a representation of the element of the set $A\{1',2',3',{1^{k}}'\}$ obtained in the equation $(\ref{1,2,3,1m})$, we obtain
    \begin{equation*}
    \begin{split}
        &(WAW)^{\dagger}+ (I_m-(WAW)^{\dagger}WAW)U(WAW)^{\dagger} +(I_m-((WAW)^{\dagger}+ (I_m \\ &- (WAW)^{\dagger}WAW)U(WAW)^{\dagger})WAW)A^{D,W}WAW(WAW)^{\dagger}\\
        &=(WAW)^{\dagger}+(I_m-(WAW)^{\dagger}WAW)U(WAW)^{\dagger} \\& +(I_m-(WAW)^{\dagger}WAW)A^{D,W}WAW(WAW)^{\dagger}\\ &-(I_m-(WAW)^{\dagger}WAW)U(WAW)^{\dagger}WAWA^{D,W}WAW(WAW)^{\dagger}\\
        &=A^{\text{\textcircled{$-$}},w}+(I_m-(WAW)^{\dagger}WAW)U((WAW)^{\dagger} \\&-(WAW)^{\dagger}WAWA^{D,W}WAW(WAW)^{\dagger}).
    \end{split}
    \end{equation*}
    Likewise, representation for $A\{1',2',4',{{}^{k}1}'\}$ can be obtained.
\end{proof}
The following result is valid for all rectangular matrices $A \in \mathbb{C}^{n \times m}_r, B \in \mathbb{C}^{m \times t}_s$ and needed for the theorem that follows. The proof is included for completeness.
\begin{lem} \label{AWB=0}
    Let $A \in \mathbb{C}^{n \times m}_r, B \in \mathbb{C}^{m \times t}_s$. If $AUB=0$ for all $U \in \mathbb{C}^{m \times m} $, then $A=0$ or $B=0$.
\end{lem}
\begin{proof}
    Using the rank normal form of $A$ and $B$, there exists $P_1 \in \mathbb{C}^{n \times n}$, $P_2 \in \mathbb{C}^{m \times m}$, $Q_1 \in \mathbb{C}^{m \times m}$ and $Q_2 \in \mathbb{C}^{t \times t}$ such that $P_1, Q_1, P_2, Q_2$ are invertible matrices and
    \begin{equation*}
        A= P_1 \begin{bmatrix}
    I_r & 0 \\
    0 & 0 
    \end{bmatrix} Q_1, \ \text{and} \  B= P_2 \begin{bmatrix}
    I_s & 0 \\
    0 & 0 
    \end{bmatrix} Q_2.
    \end{equation*}
    Let $A \neq 0$, $B \neq 0$ and $AUB=0 $ for all $U \in \mathbb{C}^{m \times m} $. Consider
    \begin{equation*}
         AQ_1^{-1}P_2^{-1}B =P_1 \begin{bmatrix}
    I_{min(r,s)} & 0 \\
    0 & 0 
    \end{bmatrix} Q_2 \neq 0,
    \end{equation*}
    which contradicts the hypothesis, hence $A=0$ or $B=0$.
\end{proof}
The following theorem gives the sufficient and necessary condition for the uniqueness of the $\{1',2',3',{1^{k}}'\} $-inverse of $A$.
\begin{thm}
    Let $A \in \mathbb{C}^{m \times n}, \ W \in \mathbb{C}^{n \times m}$ such that $k= \text{ind}(A,W) \leq 1$. Consequently, the following are equivalent
    \begin{enumerate}
        \item[(i)] $k \leq 1$.
        \item[(ii)] $A$ has a unique $\{1',2',3',{1^{k}}'\}$-inverse.
        \item[(iii)] $A$ has a unique $\{1',2',4',{{}^{k}1}'\}$-inverse.
    \end{enumerate}
\end{thm}
\begin{proof} 
     $(i) \implies (ii)$. It follows by Corollary $\ref{1,2,3,11}$.\\
     $(ii) \implies (i)$. Let $A$ have unique $\{1',2',3',{1^{k}}'\}$-inverse $A^{\text{\textcircled{$-$}},w} \in A\{1',2',3',{1^{k}}'\}$. By Theorem \ref{exist amd rep}, 
     \begin{equation*}
     \begin{split}
        A^{\text{\textcircled{$-$}},w}&+(I_m-(WAW)^{\dagger}WAW) U ((WAW)^{\dagger} \\&-(WAW)^{\dagger}WAWA^{D,W}WAW(WAW)^{\dagger}) \in A\{1',2',3',{1^{k}}'\}
        \end{split}
     \end{equation*} for all $U \in \mathbb{C}^{m \times m}$. Hence both should be equal, and we get
     \begin{equation*}
     \begin{split}
         (I_m{-}(WAW)^{\dagger}WAW)U((WAW)^{\dagger}{-}(WAW)^{\dagger}WAWA^{D,W}WAW(WAW)^{\dagger}) \\ {=}0.
         \end{split}
     \end{equation*}
     Then by Lemma $\ref{AWB=0}$, either $I_m=(WAW)^{\dagger}WAW$ or $$(WAW)^{\dagger}= (WAW)^{\dagger}WAWA^{D,W}WAW(WAW)^{\dagger}.$$ $I_m=(WAW)^{\dagger}WAW$ implies $AW$ and $WA$ are invertible matrices, hence $k=0$. In the latter case,
     \begin{equation*}
         \begin{split}
         WAW&=WAW(WAW)^{\dagger}WAW\\
             &= WAW((WAW)^{\dagger}WAWA^{D,W}WAW(WAW)^{\dagger})WAW\\
             &= WAWA^{D,W}WAW.
         \end{split}
     \end{equation*}
     Since $r(WAW)=r(A)$, using Corollary \ref{A(1,2,3)=waw(1,2,3)}, $A=AWA^{D,W}WA $ and $A^{D,W}=A^{\#,W}$. Thus, $k \leq 1$.\\
     $(i)\iff (iii)$ is similar to $(i) \iff (ii)$.
\end{proof}
Now two canonical representations will be derived for $\{1',2',3',{1^{k}}'\}$-inverse using singular value decomposition and core-nilpotent decomposition.  For the first representation, in the following theorem, we recall a canonical representation of $A^{D,W}$ of $A \in \mathbb{C}^{m \times n}$, obtained using the singular value decompositions (SVD) of $A \in \mathbb{C}^{m \times n}_r$ and $W \in \mathbb{C}^{n \times m}_s$ given as follows \cite{ben2003generalized,horn2012matrix}
\begin{equation} \label{20}
    A=K\begin{bmatrix}
        \Sigma_1 & 0 \\
        0 &0 
    \end{bmatrix}L^*  \ \text{and} \ W=M \begin{bmatrix}
        \Sigma_2 & 0 \\
        0 &0 
    \end{bmatrix}N^*,
\end{equation}
where $\Sigma_1$ and $\Sigma_2$ are diagonal matrices with diagonal entries as the singular values of $A$ and $W$, respectively. $K \in \mathbb{C}^{m \times m},\ L \in \mathbb{C}^{n \times n}, \ M \in \mathbb{C}^{n \times n} \ \text{and}  \ N\in \mathbb{C}^{m \times m}$ are unitary matrices. 
\begin{thm} \cite{meng2017dmp}
    Let $A \in \mathbb{C}^{m \times n}_r$ and $W \in \mathbb{C}^{n \times m}_s$ having SVD as in $(\ref{20})$. Then
    \begin{equation}  \label{aDW}
        {A^{D,W}}{=}K\begin{bmatrix}
            {\Sigma_1 M_1[(\Sigma_2 K_1 \Sigma_1 M_1)^D]^2} & {\Sigma_1 M_1[(\Sigma_2 K_1 \Sigma_1 M_1)^D]^3 \Sigma_2 K_1 \Sigma_1 M_2} \\
            0 & 0
        \end{bmatrix} {M^*},
    \end{equation}
    where \begin{equation*}
    N^*K= \begin{bmatrix}
    K_1 & K_2 \\
    K_3 & K_4
\end{bmatrix}, \ L^*M= \begin{bmatrix}
    M_1 & M_2 \\
    M_3 & M_4
\end{bmatrix}, \ K_1 \in \mathbb{C}^{s \times r}, \  \text{and} \ M_1 \in \mathbb{C}^{r \times s}.
\end{equation*}
\end{thm}
A canonical form for the $\{1',2',3',{1^{k}}'\}$-inverse of $A$ is now provided using the singular value decompositions of $A \in \mathbb{C}^{m \times n}_r$ and $W \in \mathbb{C}^{n \times m}_s$.
\begin{thm} \label{thm 3.10}
Let $A \in \mathbb{C}^{m \times n}_r$ and $W \in \mathbb{C}^{n \times m}_s$ having SVD as in $(\ref{20})$. Then
\begin{equation} \label{particular elt}
    A^{\text{\textcircled{$-$}},W}=N\begin{bmatrix}
        B^{\dagger}+(I_s-B^{\dagger}B)K_1\Sigma
        _1M_1B^DB^{\dagger} &0 \\ 0 & 0
    \end{bmatrix}M^*.
\end{equation}
Furthermore,  any element $X \in A\{1',2',3',{1^{k}}'\}$ is given by
\begin{equation} \label{full set 2}
    X=A^{\text{\textcircled{$-$}},W}+N\begin{bmatrix}
        (I_s-B^{\dagger}B)U_1B^{\dagger}(I-BK_1\Sigma_1M_1B^DB^{\dagger}) & 0 \\
        I_{m-s}U_3B^{\dagger}(I-BK_1\Sigma_1M_1B^DB^{\dagger}) & 0
    \end{bmatrix}M^*,
\end{equation}
where \begin{equation*}
\begin{split}
    N^*K= \begin{bmatrix}
    K_1 & K_2 \\
    K_3 & K_4
\end{bmatrix}, \ L^*M= \begin{bmatrix}
    M_1 & M_2 \\
    M_3 & M_4
\end{bmatrix}, \ K_1 \in \mathbb{C}^{s \times r}, \   \ M_1 \in \mathbb{C}^{r \times s}, \\ 
\ B= \Sigma_2K_1\Sigma_1M_1\Sigma_2 , \  U_1 \in \mathbb{C}^{s \times s} \ \text{and} \  U_3 \in \mathbb{C}^{m-s \times s}.
\end{split}
    \end{equation*}
\end{thm}
\begin{proof}
Using equation (\ref{20}), we get
\begin{equation} \label{24}
\begin{split}
WAW&=M\begin{bmatrix}
    \Sigma_2K_1\Sigma_1M_1\Sigma_2 & 0\\
    0 & 0
\end{bmatrix}N^*, \ \text{which implies} \ \\
    (WAW)^{\dagger}&=N\begin{bmatrix}
    (\Sigma_2K_1\Sigma_1M_1\Sigma_2)^{\dagger} & 0\\
    0 & 0
\end{bmatrix}M^*.
\end{split}
\end{equation}
From Remark $(\ref{remark 13})$, 
\begin{equation}\label{24a}
    A^{\text{\textcircled{$-$}},W}=(WAW)^{\dagger}+ (I_m-(WAW)^{\dagger}WAW)A^{D,W}WAW(WAW)^{\dagger}.
\end{equation} By substituting the representations for $(WAW)^{\dagger}$ and $A^{D,W}$ from equations $(\ref{24})$ and $(\ref{aDW})$ respectively, into the equation (\ref{24a}), we get required representation  of $A^{\text{\textcircled{$-$}},W}$ given in equation $(\ref{particular elt})$, where $B= \Sigma_2K_1\Sigma_1M_1\Sigma_2$. By Theorem $\ref{full set}$, any element $X \in A\{1',2',3',{1^{k}}'\}$ can be written as 
\begin{equation}\label{24b}
\begin{split}
    X= A^{\text{\textcircled{$-$}},w}&+(I_m-(WAW)^{\dagger}WAW)U((WAW)^{\dagger} \\&-(WAW)^{\dagger}WAWA^{D,W}WAW(WAW)^{\dagger}).
\end{split}
\end{equation}
Consider the partition of $U$ as $U=N\begin{bmatrix}
    U_1 & U_2 \\
    U_3 & U_4
\end{bmatrix}N^{*}$, where $U_1 \in \mathbb{C}^{s \times s}$. Now by substituting the representations of $WAW, \  (WAW)^{\dagger}, \ A^{D,W}$ from equations (\ref{24}) and (\ref{aDW}) respectively and the partition form of $U$ into the equation (\ref{24b}), we get the required equation (\ref{full set 2}).
\end{proof}

For $A\in \mathbb{C}^{n \times n}$, the core-nilpotent decomposition \cite{wang2018generalized} of $A$ is given by $ A=P\begin{bmatrix}
    C & 0\\
    0 & N
\end{bmatrix}P^{-1}$, where $P, \ C$ are nonsingular matrices and $N$ is nilpotent matrix of index $k=\text{ind}(A)$. Now we recall the following result from \cite{wei2003integral}, which gives a canonical representation for $A^{D,W}$ of $A \in \mathbb{C}^{m \times n}$, obtained using the core-nilpotent decompositions of $AW$ and $WA$. 
\begin{thm}\cite{wei2003integral}
    Let $A \in \mathbb{C}^{m \times n}$, $W\in \mathbb{C}^{n \times m}$ and $k=\text{ind}(A,W)$. Then we have
    \begin{equation} \label{2003AMC}
    \begin{split}
    A=S\begin{bmatrix}
        A_1 &0 \\
        0 & A_2
    \end{bmatrix}R^{-1}, \ W =R\begin{bmatrix}
        W_1 & 0\\
        0 & W_2
    \end{bmatrix}S^{-1} \ \text{and}, \\ \ A^{D,W}=S\begin{bmatrix}
        (W_1A_1W_1)^{-1} &0 \\
        0 & A_2
    \end{bmatrix}R^{-1},
    \end{split}
\end{equation}
where $R,S,A_1,W_1$ are nonsingular matrices.
\end{thm}

The following theorem gives another canonical form for the elements of the set $A\{1',2',3',{1^{k}}'\}$ using the representations of $A$ and $W$ given in equation (\ref{2003AMC}) by using the core-nilpotent decompositions of $AW$ and $WA$ and SVD of $WAW$. 
\begin{thm} \label{3.12}
    Let $A \in \mathbb{C}^{m \times n}, \ W \in \mathbb{C}^{n \times m}$ such that $k= \text{ind}(A,W)$ and $r(WAW)=r(A)=r$. Let $A, W$ have the form as given in equation $(\ref{2003AMC})$ and let SVD of $WAW$ be as follows,
    \begin{equation} \label{svd waw}
    WAW=P\begin{bmatrix}
        B_1 & 0 \\
        0 & 0
    \end{bmatrix}Q^{-1},
    \end{equation} where $P,Q$ are unitary and $B_1$ is a diagonal matrix with positive entries. Then

\begin{equation}
A^{\text{\textcircled{$-$}},W}=Q\begin{bmatrix}
       {B_1}^{-1} & 0\\
    Q_3(W_1A_1W_1)^{-1}R_1+Q_4A_2R_3&  0
    \end{bmatrix}P^{-1}.
\end{equation}
    Furthermore, any element $X \in A\{1',2',3',{1^{k}}'\}$ is given by
    \begin{equation} \label{15}
        X=A^{\text{\textcircled{$-$}},W}+Q\begin{bmatrix}
            0 &0\\
            E & 0
        \end{bmatrix}P^{-1},
    \end{equation}
    such that $E=U_3(B_1^{-1}-Q_1(W_1A_1W_1)^{-1}R_1)-U_4Q_3(W_1A_1W_1)^{-1}R_1$,  $Q^{-1}S=\begin{bmatrix}
        Q_1 & Q_2\\
        Q_3 &Q_4
    \end{bmatrix}$, $R^{-1}P=\begin{bmatrix}
        R_1 & R_2\\
        R_3 &R_4
    \end{bmatrix}$ and $Q^{-1}UQ=\begin{bmatrix}
        U_1 & U_2\\
        U_3 &U_4
    \end{bmatrix}$,
    where $U \in \mathbb{C}^{m \times m}$ and $Q_1,R_1,U_1 \in \mathbb{C}^{r \times r}$.
\end{thm}
\begin{proof}
    For $A \in \mathbb{C}^{m \times n}, \ W \in \mathbb{C}^{n \times m}$, using the representation $WAW$ as in equation \eqref{svd waw}. we get
    \begin{equation} \label{waw dagger}(WAW)^{\dagger}=Q\begin{bmatrix}
            B_1^{-1} & 0\\
            0 & 0
        \end{bmatrix}P^{-1},
    \end{equation}
Thus, 
\begin{equation} \label{waw(waw)dagger}
 (WAW)^{\dagger}(WAW)=Q\begin{bmatrix}
        I_{r} & 0 \\
        0 & 0
\end{bmatrix}Q^{-1} \ \text{and} \  WAW(WAW)^{\dagger}=P\begin{bmatrix}
        I_{r} & 0 \\
        0 & 0
\end{bmatrix}P^{-1}.
\end{equation}
    Further, let the core-nilpotent decompositions of $WA$ and $AW$ respectively are $WA=R\begin{bmatrix}
        C_2 & 0 \\
        0 & N_2
    \end{bmatrix}R^{-1}$ and $AW=S\begin{bmatrix}
        C & 0 \\
        0 & N
    \end{bmatrix}S^{-1}$, 
    where $N_2, \ N$ are nilpotent matrices of order less than equal to $k$, where $k=\text{ind}(A,W)$. So, by \cite{wei2003integral} $A$, $W$ and $A^{D,W}$ can be written as follows
\begin{equation}  \label{ADW 19}
\begin{split}
    A=S\begin{bmatrix}
        A_1 &0 \\
        0 & A_2
    \end{bmatrix}R^{-1}, \ W =R\begin{bmatrix}
        W_1 & 0\\
        0 & W_2
    \end{bmatrix}S^{-1} \ \text{and}, \\ A^{D,W}=S\begin{bmatrix}
        (W_1A_1W_1)^{-1} &0 \\
        0 & A_2
    \end{bmatrix}R^{-1},
\end{split}
\end{equation}
where $R, \ S \ ,A_1, \ W_1$ are nonsingular matrices such that $r(A_1)=r(W_1)=r(C_2)=r(C)$. Now,
\begin{equation}
\begin{split}
    (I_m-&(WAW)^{\dagger}(WAW))A^{D,W} WAW(WAW)^{\dagger}  = \\ & Q\begin{bmatrix}
        0 &0 \\
        0 & I_{m-r}
    \end{bmatrix}Q^{-1} S\begin{bmatrix}
        (W_1A_1W_1)^{-1} &0 \\
        0 & A_2
    \end{bmatrix}R^{-1} P\begin{bmatrix}
        I_{r} & 0 \\
        0 & 0 
    \end{bmatrix}P^{-1}.
    \end{split}
\end{equation}
Now, write $Q^{-1} S$ and $R^{-1} P$ in a partitioned form,
\begin{equation} \label{QS RP}
    Q^{-1} S=\begin{bmatrix}
        Q_1 & Q_2 \\
        Q_3 & Q_4
    \end{bmatrix} \ \text{and} \ R^{-1} P= \begin{bmatrix}
        R_1 & R_2 \\
        R_3 & R_4
    \end{bmatrix},
\end{equation}
where $Q_1, R_1 \in \mathbb{C}^{r \times r}$. Thus, 
\begin{equation*}
    \begin{split}
        A^{\text{\textcircled{$-$}},W} &=(WAW)^{\dagger} + (I_m-(WAW)^{\dagger}(WAW))A^{D,W}WAW(WAW)^{\dagger} \\&=Q \begin{bmatrix}
    {B_1}^{-1} & 0\\
    Q_3(W_1A_1W_1)^{-1}R_1+Q_4A_2R_3&  0
\end{bmatrix}P^{-1}.
    \end{split}
\end{equation*} Now, by Theorem $\ref{full set}$, any element $X \in A\{1',2',3',{1^{k}}'\}$ can be written as 
\begin{equation}\label{24c}
\begin{split}
    X= A^{\text{\textcircled{$-$}},W}+(I_m-(WAW)^{\dagger}WAW)U((WAW)^{\dagger} \\-(WAW)^{\dagger}WAWA^{D,W}WAW(WAW)^{\dagger}).
    \end{split}
\end{equation}
Consider the matrix $U$ in the partitioned form as $QUQ^{-1}=\begin{bmatrix}
    U_1 & U_2 \\
    U_3 & U_4
\end{bmatrix}$, where $U_1 \in \mathbb{C}^{r \times r}$. 
Now by substituting the representations of $(WAW)^{\dagger}$, $ (WAW)^{\dagger}WAW$, $ WAW(WAW)^{\dagger}$, $A^{D,W}$ from equations \eqref{waw dagger},  \eqref{waw(waw)dagger},  \eqref{ADW 19} respectively and the partition form of $U$ along with the representations in \eqref{QS RP} into the equation (\ref{24c}), we get the required equation (\ref{15}).
\end{proof} 
\begin{rem}
    Note that in Theorem \ref{thm 3.10} and Theorem \ref{3.12}, if we take $m=n$ and $W=I$, corresponding representations for $\{1,2,3,1^k\}$-inverses can be obtained.
\end{rem}

\section{Characterizations and properties of $\{1',2',3', {1^{k}}' \}$-inverses}\label{sec4}
This section presents various characterizations and properties of the $W$-weighted $\{1,2,3,1^k\}$-inverses. Further, a new representation of $A^{D,W}$ is obtained using the characterization of $\{1',2',3',{1^{k}}'\}$-inverse. The following theorem entails that matrix $A$ cannot be uniquely determined by its $\{1',2',3',{1^{k}}'\}$-inverse.
\begin{thm}
Let $A \in \mathbb{C}^{m \times n}, \ W \in \mathbb{C}^{n \times m}$ such that $k= \text{ind}(A,W)$. If $X$ is a $\{1',2',3',{1^{k}}'\}$ of $A$, then $X$ is also a $\{1',2',3',{1^{k}}'\}$-inverse of $A+(AW)^{k+1}(XW)^k(I_n-WXWA)$.
\end{thm}
\begin{proof}
    Let $B=A+(AW)^{k+1}(XW)^k(I_n-WXWA)$, then $BWX= AWX+(AW)^{k+1}(XW)^k(I_n-WXWA)WX=AWX$. This gives,
    \begin{equation*}
        \begin{split}
            BWXWB&=AWXWB=AWXW(A+(AW)^{k+1}(XW)^k(I_n-WXWA)) \\&=B.
        \end{split}
    \end{equation*}
    Clearly, $XWBWX=XWAWX=X$ and $(WBWX)^* = (WAWX)^* = WAWX = WBWX$. Hence, $X \in B\{1',2',3'\}$.
    By using induction on $p\in \mathbb{N}$, we obtain
    \begin{equation*}
        (BW)^p=(AW)^p+\sum_{j=0}^{p-1}(AW)^{k+p-j}(XW)^{k}(I_n-WXWA)W(AW)^j.
    \end{equation*}
   For $p=k+1$, 
    \begin{equation*}
        (BW)^{k+1}=(AW)^{k+1}+\sum_{j=0}^{k-1}(AW)^{2k+1-j}(XW)^{k}(I_n-WXWA)W(AW)^j,
    \end{equation*}
    and hence, we get $XW(BW)^{k+1}=(BW)^k$. Thus $X\in B\{1', 2', 3', {1^{k}}'\}$.
\end{proof}
In \cite{wu20221}, it was demonstrated that the two matrices $A$ and $B$ are equal if they have a common $\{1,2,3\}$ and $\{1,2,4\}$-inverse. Next, an example is provided to show that it is not true for the $\{1',2',3'\}$ and $\{1',2',4'\}$-inverses.
\begin{ex} \label{ex A not equal B}
    Take $A=\begin{bmatrix}
        0 & 1
    \end{bmatrix}, \ B=\begin{bmatrix}
        b & 1
    \end{bmatrix} , \
    W=\begin{bmatrix}
        0 \\
        1
    \end{bmatrix}$, where $b \in \mathbb{C}$. Clearly, $X=\begin{bmatrix}
        0 & 1
    \end{bmatrix} \in A\{ 1',2',3'\} \cap B\{ 1',2',3'\}$ and $X \in A\{ 1',2',4'\} \cap B\{ 1',2',4'\}$ for all $b \in \mathbb{C}$, but $A \neq B$ for $b \neq 0$.
    \end{ex}
\begin{lem} \label{lemma 28}
Let $A,B \in \mathbb{C}^{m \times n}, \ W \in \mathbb{C}^{n \times m}$. If $A\{1',2',3'\} \cap B\{1',2',3'\}\neq \emptyset$ and $A\{1',2',4'\} \cap B\{1',2',4'\}\neq \emptyset$, then $A$ and $B$ belong to same equivalence class of the equivalence relation $\sim$, where $A \sim B$ if and only if $WAW=WBW$ for $A,B \in \mathbb{C}^{m \times n}$.
\end{lem}
\begin{proof}
    Suppose that $X \in A\{1',2',3'\} \cap B\{1',2',3'\}$, therefore
    \begin{equation*}
        \begin{split}
            WA&=WAWXWA =(WAWX)^*WA=X^*(WAW)^*WA\\&=(XWBWX)^*(WAW)^*WA=(WBWX)^*X^*(WAW)^*WA\\&=WBWX(X^*(WAW)^*WA)=WBWX(WAWXWA) =WBWXWA.
        \end{split}
    \end{equation*}
    
    Let $Y \in A\{1',2',4'\} \cap B\{1',2',4'\}$, then
    \begin{equation*}
        \begin{split}
            AW&=AWYWAW=AW(YWAW)^*=AW(WAW)^*Y^*\\&=AW(WAW)^*(YWBWY)^*=AW(WAW)^*Y^*(YWBW)^* \\& =AWYWAWYWBW=AWYWBW.
        \end{split}
    \end{equation*}
    Hence, it follows
    \begin{equation*}
        \begin{split}
            WAW&=WAWYWBW=WBWXWAWYWBW\\&=WBWXWBWXWAWYWBWXWBW\\&=WBWXWAWXWBW=WBWXWBW=WBW.
        \end{split}
    \end{equation*}
\end{proof}
In Example $\ref{ex A not equal B}$, it can be seen that though $A\neq B$, for the given $W$, $WAW=WBW$ holds.
The following result gives sufficient conditions for a matrix $A$ to be uniquely determined by its $\{1',2',3',{1^{k}}'\}$-inverse upto the equivalence relation $\sim$, i.e., multiplication by $W$ on both sides.
\begin{thm}
 \label{waw=wbw} $\displaystyle$
Let $A,B \in \mathbb{C}^{m \times n}, \ W \in \mathbb{C}^{n \times m}$. If $A$ and $B$ have an identical $\{1',2',3',{1^{k}}'\}$-inverse and an identical $\{1',2',4',{}^{l'}1\}$-inverse, then $A \sim B$ i.e. $WAW=WBW$.
\end{thm}
\begin{proof}
    The proof follows from Lemma \ref{lemma 28}.
\end{proof}
The following theorem investigates the class of matrices for which the sets $ A\{1',2',3'\}$ and $A\{1',2',3',{1^{k}}'\} $ are the same.
\begin{thm}
 \label{AW inv}
Let $A \in \mathbb{C}^{m \times n}, \ W \in \mathbb{C}^{n \times m}$ such that $n>m$ and $k \geq 0$ be some integer. Then the following statements are equivalent,
\begin{enumerate}
    \item[(i)] $AW$ is invertible or nilpotent.
    \item[(ii)] $ A\{1',2',3'\} =A\{1',2',3',{1^{k}}'\} $.
\end{enumerate}
\end{thm}
\begin{proof}
    $(i) \implies (ii)$. If $AW$ is nilpotent then $ X \in A\{1',2',3'\}$ if and only if $X \in A\{1',2',3',{1^{k}}'\}$, since $(XW)(AW)^{k+1}=(AW)^k$ holds trivially. If $AW$ is invertible, then to verify $(XW)(AW)^{k+1}=(AW)^k$, it is enough to show $(XWAW)=I$. Now, if $X \in A\{1',2',3'\}$ then $AWXWA=A$ implies $AWXWAW=AW$. Further, by pre-multiplying both sides of $AWXWAW=AW$ by $(AW)^{-1}$, we get $XWAW=I$.\\
    $(ii)\implies (i)$. Let $ A\{1',2',3'\} =A\{1',2',3',{1^{k}}'\} $. By Theorem $\ref{lemma 7}$, $X=(WAW)^{\dagger}+ (I_m-(WAW)^{\dagger}WAW)U(WAW)^{\dagger} $ is an element of $A\{1',2',3'\}$ for all $U \in \mathbb{C}^{m \times m}$. It will satisfy equation $({1^{k}}')$ by hypothesis, i.e.,
    \begin{equation} \label{1^k}
        ((WAW)^{\dagger}+ (I_m-(WAW)^{\dagger}WAW)U(WAW)^{\dagger})W(AW)^{k+1}=(AW)^k.
    \end{equation}
    In particular, for $U=0$ in equation $(\ref{1^k})$ gives $(WAW)^{\dagger})W(AW)^{k+1}=(AW)^k$. Now, by substituting $(AW)^k$ for $(WAW)^{\dagger})W(AW)^{k+1}$ into the equation $(\ref{1^k})$, we get
    \begin{equation*}
        (I_m-(WAW)^{\dagger}WAW)U(WAW)^{\dagger})W(AW)^{k+1}=0,
    \end{equation*}
    for all $U \in \mathbb{C}^{m \times m}$.
Thus by Lemma $\ref{AWB=0}$, either $(I_m-(WAW)^{\dagger}WAW)=0$ or $(AW)^k=0$ which gives $AW$ is invertible or nilpotent.
\end{proof}
\begin{rem}
    The Theorem $\ref{AW inv}$ can be easily verified for the matrices $B$ and $W$ as given in Example $\ref{ex A not equal B}$. Clearly $BW$ is invertible, hence if we take $b\neq 0$, then $X\in B\{{1^{k}}'\}$.  Moreover, since $X \notin B\{{{}^{k}1}'\}$, further investigation is needed for the class of matrices for which the sets $ B\{1',2',4'\}$ and $B\{1',2',4',{{}^{k}1}'\} $ are the same.
\end{rem}
In the following theorem, we investigate the class of matrices for which the sets $ A\{1',2',4'\}$ and $A\{1',2',4',{{}^{k}1}'\} $ are the same.
\begin{thm} \label{WA inv}
Let $A \in \mathbb{C}^{m \times n}, \ W \in \mathbb{C}^{n \times m}$ such that $m>n$ and $k \geq 0$ be some integer. Then the following statements are equivalent,
\begin{enumerate}
    \item[(i)] $WA$ is invertible or nilpotent.
    \item[(ii)] $ A\{1',2',4'\} =A\{1',2',4',{{}^{k}1}'\} $.
\end{enumerate}
\end{thm}
\begin{proof}
    The proof follows similarly as in Theorem $\ref{AW inv}$.
\end{proof}
The following example verifies Theorem $\ref{WA inv}$ and illustrates the fact that Theorem $\ref{AW inv}$ and Theorem $\ref{WA inv}$ are independent.
\begin{ex} \label{ex last}
    Take $A=\begin{bmatrix}
        b \\
        1
    \end{bmatrix}, \ W=\begin{bmatrix}
        0 & 1
    \end{bmatrix} \  $, where $b \in \mathbb{C}$ and $b\neq 0$. Clearly, $r(WAW)=r(A)$. It is easy to see that $WAW=\begin{bmatrix}
        0 & 1
    \end{bmatrix}$ . So, $X=\begin{bmatrix}
        0 \\
        1
    \end{bmatrix} \in A\{ 1',2',3'\} $ and also $X \in A\{ 1',2',4'\} $. Also, $WA=[1]$, which is invertible. Evidently, $X\in A\{{{}^{k}1}'\}$ but $X \notin A\{{1^{k}}'\}$.
    \end{ex}

Subsequently, the following results focus on several characterizations of the set $A\{1',2',3', {1^{k}}'\}$.
 
\begin{thm} \label{A2ts (i)} 
Let $A \in \mathbb{C}^{m \times n}, \ W \in \mathbb{C}^{n \times m}$ and $k= \text{ind}(A,W)$. Then 
\begin{itemize} 
    \item[(i)] $X \in A\{1',2',3', {1^{k}}'\}$ if and only if $XWAWX=X$, $r(WAW)=r(A)$, $\mathcal{N}(X)=\mathcal{N}((WA)^*)$ and $\mathcal{R}(X) \supseteq \mathcal{R}((AW)^k)$.
    \item[(ii)]  $X \in A\{1',2',3', {1^{k}}'\}$ if and only if $AWXWA=A$, $r(WAW)=r(A)$, $\mathcal{N}(X)=\mathcal{N}((WA)^*)$ and $\mathcal{R}(X) \supseteq \mathcal{R}((AW)^k)$.
\end{itemize}

\end{thm}
\begin{proof}
    Let $X \in A\{1',2',3', {1^{k}}'\}$. In view of Theorem $\ref{a2ts}$, to prove the `if' part of $(i)$ and $(ii)$, it is enough to show that $\mathcal{R}(X) \supseteq \mathcal{R}((AW)^k)$. From equation (${1^{k}}'$), $XW(AW)^{k+1}=(AW)^k$ gives $\mathcal{R}(X) \supseteq \mathcal{R}((AW)^k)$. Conversely, if $\mathcal{R}(X) \supseteq \mathcal{R}((AW)^k)$ then there exists some matrix $U \in \mathbb{C}^{n \times m}$ such that $(AW)^k=XU$. Hence $XW(AW)^{k+1}=XWAWXU=XU=(AW)^k.$ Other conditions of the `only if' part for the proof $(i)$ and $(ii)$ are direct from Theorem $\ref{a2ts}$.
\end{proof}
Particularly, if $m=n$ and $W=I$ in Theorem \ref{A2ts (i)}, a characterizations of $\{1,2,3, 1^{k}\}$-inverse of $A$ are given in the following corollaries.
\begin{cor} \cite{wu20221}
    Let $A \in \mathbb{C}^{n \times n}$, and $k= \text{ind}(A)$. Then $X \in A\{1,2,3, 1^{k}\}$ if and only if $XAX=X$, $\mathcal{N}(X)=\mathcal{N}((A)^*)$ and $\mathcal{R}(X) \supseteq \mathcal{R}((A)^k)$.
\end{cor}
\begin{cor} 
    Let $A \in \mathbb{C}^{n \times n}$, and $k= \text{ind}(A)$. Then $X \in A\{1,2,3, 1^{k}\}$ if and only if $AXA=X$, $\mathcal{N}(X)=\mathcal{N}((A)^*)$ and $\mathcal{R}(X) \supseteq \mathcal{R}((A)^k)$.
\end{cor}

\noindent The following theorem gives some properties of $\{1',2',3',{1^{k}}'\}$-inverses, which will be further used to get another characterization of the set $A\{1',2',3',{1^{k}}'\}$.
\begin{thm} \label{proj}
Let $A \in \mathbb{C}^{m \times n}, \ W \in \mathbb{C}^{n \times m}$ and $k= \text{ind}(A,W)$. Then for $X \in A\{1',2',3',{1^{k}}'\}$
\begin{enumerate}
    \item[(i)] $WAWX=P_{\mathcal{R}(WA)}$,
    \item[(ii)] $XWAW=P_{\mathcal{R}(X), \mathcal{N}(AW)}$.
\end{enumerate}
\end{thm}
\begin{proof}
    \begin{itemize}
        \item[(i)] For $X {\in} A\{1',2',3',{1^{k}}'\} $, clearly $WAWX{=}P_{\mathcal{R}(WAW)}$. Now, since $r(WAW){=}r(WA)$ and $\mathcal{R}(WAW) {\subseteq} \mathcal{R}(WA)$ therefore,  $\mathcal{R}(WAW) {=} \mathcal{R}(WA)$. 
        \item[(ii)] For $X {\in} A\{1',2',3',{1^{k}}'\} $, clearly $XWAW{=}P_{\mathcal{R}(X), \mathcal{N}(WAW)}$. By $r(WAW)=r(AW)$ and $\mathcal{N}(WAW) {\supseteq} \mathcal{N}(AW)$, we get $\mathcal{N}(WAW)=\mathcal{N}(AW)$. 
    \end{itemize}
\end{proof}
The following result is obtained when choosing $W=I$ and $m=n$ in Theorem \ref{proj}.
\begin{cor}
    Let $A \in \mathbb{C}^{n \times n}$, such that $k= \text{ind}(A)$. Then for $X \in A\{1,2,3,1^{k}\}$
\begin{enumerate}
    \item[(i)] $AX=P_{\mathcal{R}(A)}$
    \item[(ii)] $XA=P_{\mathcal{R}(X), \mathcal{N}(A)}$.
\end{enumerate}
\end{cor}

The following theorem characterizes $\{1',2',3',{1^{k}}'\}$-inverses from a geometrical point of view.
\begin{thm} \label{30}
Let $A \in \mathbb{C}^{m \times n}, \ W \in \mathbb{C}^{n \times m}$ such that $k= \text{ind}(A,W)$ and $r(WAW)=r(A)$. Then $X$ satisfies the following system of equations
\begin{equation*}
    XWAWX=X, \  WAWX=P_{\mathcal{R}(WA)}, \ \mathcal{R}(X) \supseteq \mathcal{R}((AW)^k),
\end{equation*}
if and only if $X \in A\{1',2',3',{1^{k}}'\}$.
\end{thm}
\begin{proof}
    Proof follows from Theorem $\ref{A2ts (i)}$ and Theorem $\ref{proj}$.
\end{proof}
In particular, for $m=n$ and $W=I$ in the Theorem \ref{30}, a new characterization of $\{1,2,3, 1^{k}\}$-inverse of $A$ is obtained.
\begin{cor}
    Let $A \in \mathbb{C}^{n \times n}$, such that $k= \text{ind}(A)$. Then $X \in \mathbb{C}^{n \times n}$ satisfies the following system of equations 
\begin{equation*}
    XAX=X, \ AX=P_{\mathcal{R}(A)}, \ \mathcal{R}(X) \supseteq \mathcal{R}((A)^k),
\end{equation*}
if and only if $X \in A\{1,2,3,1^{k}\}$.
\end{cor}

The following theorem gives a characterization of the set $A\{1',2',3',{1^{k}}'\}$ involving the $W$-weighted Drazin inverse, which is further helpful in getting a representation of the $W$-weighted Drazin inverse in terms of the $\{1',2',3',{1^{k}}'\}$-inverse.
\begin{thm}
    Let $A \in \mathbb{C}^{m \times n}, \ W \in \mathbb{C}^{n \times m}$ such that $k= \text{ind}(A,W)$ and $r(WAW)=r(A)$. $X\in \mathbb{C}^{m \times n}$ satisfies
\begin{equation*}
    XWAWX=X, \  WAWX=P_{\mathcal{R}(WA)}, \ XW(AW)^k=A^{D,W}W(AW)^k,
\end{equation*}
if and only if $X \in A\{1',2',3',{1^{k}}'\}$.
\end{thm}
\begin{proof}
     If $X \in A\{1',2',3',{1^{k}}'\}$, it is enough to show that $XW(AW)^k=A^{D,W}W(AW)^k$, in view of Theorem \ref{proj}. Now, $$XW(AW)^k=XW((AW)^{k+1}A^{D,W}W)=(AW)^kA^{D,W}W=A^{D,W}W(AW)^k.$$ Conversely, $XWAWX=X, \  WAWX=P_{\mathcal{R}(WA)}$ implies $X \in A\{1',2',3'\}$. Now, $(AW)^k{=}AWXW(AW)^k{=}AWA^{D,W}W(AW)^k{=}A^{D,W}W(AW)^kAW=XW(AW)^kAW=XW(AW)^{k+1}.$
\end{proof}
In the following theorem, a representation of $A^{D,W}$ is obtained in terms of $ \{1',2',3',{1^{k}}'\} $-inverse of $A$.
\begin{thm}

Let $A \in \mathbb{C}^{m \times n}, \ W \in \mathbb{C}^{n \times m}$ such that $k = \text{ind}(A,W)$. If $X \in A\{1',2',3',{1^{k}}'\}$ then $A^{D,W}=(XW)^{k+2}(AW)^kA$.
\end{thm}
\begin{proof} Using the properties of the $W$-weighted Drazin inverse, it follows,
\begin{equation*}
        \begin{split}
            A^{D,W}W&=(AW)^k(A^{D,W}W)^{k+1}\\
            &= XW(AW)^{k+1}(A^{D,W}W)^{k+1}\\
            &=(XW)^{k+1}(AW)^{2k+1}(A^{D,W}W)^{k+1}\\
            &=(XW)^{k+1}[(AW)^{k}(A^{D,W}W)^{k+1}(AW)^{k+1}]\\
            &= (XW)^{k+1}[(A^{D,W}W)(AW)^{k+1}]\\
            &= (XW)^{k+1}(AW)^{k}.
        \end{split}
    \end{equation*}
    Now,  $A^{D,W}=A^{D,W}WAWA^{D,W}=A^{D,W}WA^{D,W}WA=(A^{D,W}W)^2A=((XW)^{k+1}(AW)^{k})^2A$. Now, since $XW(AW)^k=A^{D,W}W(AW)^k$, hence we have $ (XW)^{k+1} (AW)^k = (A^{D,W} W)^{k+1}(AW)^k$ and 
    \begin{equation*}
    \begin{split}
        A^{D,W}&=((XW)^{k+1}(AW)^{k})^2A \\&=(XW)^{k+1}(AW)^{k}(XW)^{k+1}(AW)^{k}A \\ &=(XW)^{k+1}(AW)^{k}(A^{D,W}W)^{k+1}(AW)^{k}A\\&=(XW)^{k+1}(A^{D,W}W)(AW)^{k}A\\&= (XW)^{k+1}(XW)(AW)^{k}A\\&=(XW)^{k+2}(AW)^kA.
    \end{split}
    \end{equation*}
\end{proof}

The following theorem entails some spectral properties of $XW$, where $X \in A\{{1^{k}}'\}$.
\begin{thm}
Let $A \in \mathbb{C}^{m \times n}, \ W \in \mathbb{C}^{n \times m}$ such that $k= \text{ind}(A,W)$. If $X \in A\{{1^{k}}'\}$, then any eigenvector of $AW$ with respect to the non-zero eigenvalue $\lambda$, is an eigenvector of $XW$ with respect to the eigenvalue $1/{\lambda}$.
\end{thm}
\begin{proof}
    Let $v$ be the eigenvector of $AW$ with respect to the non-zero eigenvalue $\lambda$. Then
    \begin{equation*}
    \begin{split}
        XW v &= XW\left( \left( \frac{1}{\lambda} \right)^{k+1}{\lambda}^{k+1}v\right)=XW \left( \frac{1}{\lambda} \right)^{k+1}{(AW)}^{k+1}v \\&= \left( \frac{1}{\lambda} \right)^{k+1}XW(AW)^{k+1}v=\left( \frac{1}{\lambda} \right)^{k+1}(AW)^{k}v=\left( \frac{1}{\lambda} \right)v.
    \end{split}
    \end{equation*}
\end{proof}
The following theorem gives a relation on invariant subspaces of $AW$ and $XW$ for $A  \in \mathbb{C}^{m \times n}, \ W \in \mathbb{C}^{n \times m}$ and $X \in A\{1',2',3',{1^{k}}'\}$.
\begin{thm}
Let $A,X \in \mathbb{C}^{m \times n}, \ W \in \mathbb{C}^{n \times m}$ such that $k= \text{ind}(A,W)$ and $X \in A\{1',2',3'\}$. Then $X \in A\{1',2',3',{1^{k}}'\}$ if and only if $AW \mathcal{S} = \mathcal{S} $ implies $XW\mathcal{S}=\mathcal{S}$ for any subspace $\mathcal{S}$ of $\mathbb{C}^m$.
\end{thm}
\begin{proof}
    Let $X \in A\{1',2',3',{1^{k}}'\}$ for $\mathcal{S} \subseteq  \mathbb{C}^{m}$ such that $AW\mathcal{S} = \mathcal{S}$. Then 
    \begin{equation*}
        \mathcal{S} =AW\mathcal{S} = (AW)^2\mathcal{S}=\cdots=(AW)^k\mathcal{S}=(AW)^{k+1}\mathcal{S}.
    \end{equation*}
    Now,
    \begin{equation*}
        XW\mathcal{S}=(XW)(AW)^{k+1}\mathcal{S}=(AW)^k\mathcal{S}=\mathcal{S}.
    \end{equation*}
    Conversely, let $AW\mathcal{S}=\mathcal{S}$ implies $XW\mathcal{S}=\mathcal{S}$. Since $AW\mathcal{R}((AW)^k)=\mathcal{R}((AW)^{k+1})=\mathcal{R}((AW)^k)$ and by hypothesis, $XW\mathcal{R}((AW)^k)=\mathcal{R}((AW)^k)$. Hence there exist some $U \in \mathbb{C}^{m \times m}$ such that  $(AW)^k=XW(AW)^kU$ , so
    \begin{equation*}
        \begin{split}
            (XW)(AW)^{k+1} &=XWAW(AW)^k
            = XWAWXW(AW)^kU\\&=XW(AW)^kU=(AW)^k.
        \end{split}
    \end{equation*}
    Hence $X \in A\{1',2',3',{1^{k}}'\} $.
\end{proof}

\backmatter
\section*{Statements and Declarations}

\bmhead{Funding} No funding has been provided for this research.
\bmhead{Competing Interests} The authors have no competing interets to disclose.

\bmhead{Author Contributions} Both authors contributed sufficiently to the work.




\end{document}